\documentclass[a4paper]{article}

\usepackage{hyperref}
\usepackage{authblk}
\usepackage{xypic,graphicx,enumerate,boxedminipage,amsmath,amssymb,amsthm,xspace}
\usepackage[misc,geometry]{ifsym}
\usepackage[T1]{fontenc}
\usepackage{fullpage}
\usepackage{tikz}
\usetikzlibrary{arrows,shapes,calc}
\newcommand{\dia}{\hfill{$\diamond$}}

\hypersetup{
	colorlinks=true,
	linkcolor=blue,
	filecolor=violet,  
	citecolor=violet,
	urlcolor=cyan,
	pdftitle={Computing Subset Vertex Covers in H-Free Graphs},
	pdfpagemode=FullScreen,
}

\newcommand{\vc}{\textsc{Vertex Cover}}
\newcommand{\svc}{\textsc{Subset Vertex Cover}}

\newtheorem{open}{Open Problem}
\newtheorem{theorem}{Theorem}
\newtheorem{lemma}[theorem]{Lemma}
\newtheorem{corollary}[theorem]{Corollary}
\newtheorem{definition}[theorem]{Definition}
\newtheorem{proposition}[theorem]{Proposition}
\newtheorem{remark}[theorem]{Remark}
\newtheorem{claim}{Claim}[theorem]

\newcommand{\ssi}{\subseteq_i}

\newcommand{\NP}{{\sf NP}}
\newcommand{\sP}{{\sf P}}

\newcommand{\calA}{{\cal A}}
\newcommand{\calB}{{\cal B}}
\newcommand{\calL}{{\cal L}}
\newcommand{\calR}{{\cal R}}
\newcommand{\mim}{\mathsf{mim}}
\newcommand{\cw}{\mathsf{cw}}
\newcommand{\rep}{\mathsf{rep}}
\newcommand{\best}{\mathsf{best}}
\newcommand{\we}{\omega}
\newcommand{\nec}{\mathsf{nec}}
\newcommand{\wrt}{with respect to\ }
\newcommand{\Ti}{\ensuremath{T}}
\newcommand{\Tis}{$\Ti$-independent set\xspace}
\newcommand{\Tiss}{$\Ti$-independent sets\xspace}

\newcommand{\problemdef}[3]{
	\begin{center}
		\begin{boxedminipage}{1.02\textwidth}
			\textsc{{#1}}\\[1pt]  
			\begin{tabular}{ r p{0.8\textwidth}}
				\textit{~~~~Instance:} & {#2}\\
				\textit{Question:} & {#3}
			\end{tabular}
		\end{boxedminipage}
	\end{center}
}

\definecolor{lime}{HTML}{A6CE39}
\DeclareRobustCommand{\orcidicon}{
	\begin{tikzpicture}
		\draw[lime, fill=lime] (0,0) 
		circle [radius=0.16] 
		node[white] {{\fontfamily{qag}\selectfont \tiny ID}};
		\draw[white, fill=white] (-0.0625,0.095) 
		circle [radius=0.007];
	\end{tikzpicture}
	\hspace{-2mm}
}

\foreach \x in {A, ..., Z}{
	\expandafter\xdef\csname orcid\x\endcsname{\noexpand\href{https://orcid.org/\csname orcidauthor\x\endcsname}{\noexpand\orcidicon}}
}

\title{Computing Subset Vertex Covers in $H$-Free Graphs}
\author[1]{Nick Brettell\footnote{Nick Brettell was supported by the New Zealand Marsden Fund.}\orcidA{}}
\author[2]{Jelle J. Oostveen\footnote{Jelle Oostveen was supported by the NWO grant OCENW.KLEIN.114 (PACAN).}\orcidE{}}
\author[3]{Sukanya Pandey\orcidB{}}
\author[4]{Dani\"el Paulusma\orcidC{}}
\author[5]{Johannes Rauch\footnote{Johannes Rauch was supported by the German Academic Scholarship Foundation (Studienstiftung des deutschen Volkes).}\orcidJ}
\author[2]{\mbox{Erik Jan} van Leeuwen\orcidD{}}

\affil[1]{Victoria University of Wellington, Wellington, New Zealand \texttt{nick.brettell@vuw.ac.nz}}
\affil[2]{Utrecht University, Utrecht, The Netherlands \texttt{\{j.j.oostveen,e.j.vanleeuwen\}@uu.nl}}
\affil[3]{RWTH Aachen University, Aachen, Germany \texttt{pandey@algo.rwth-aachen.de}}
\affil[4]{Durham University, Durham, United Kingdom \texttt{daniel.paulusma@durham.ac.uk}}
\affil[5]{Ulm University, Ulm, Germany \texttt{johannes.rauch@uni-ulm.de}}

\begin{document}
\maketitle

\begin{abstract} We consider a natural generalization of {\sc Vertex Cover}: the {\sc Subset Vertex Cover} problem, which is to decide for a graph $G=(V,E)$, a subset $T\subseteq V$ and integer~$k$, if $V$ has a subset $S$ of size at most $k$, such that $S$ contains at least one end-vertex of every edge incident to a vertex of $T$. A graph is $H$-free if it does not contain $H$ as an induced subgraph. 
We solve two open problems from the literature by proving
 that {\sc Subset Vertex Cover} is \NP-complete on subcubic (claw,diamond)-free planar graphs and on $2$-unipolar graphs, a subclass of $2P_3$-free weakly chordal graphs. Our results show for the first time that {\sc Subset Vertex Cover} is computationally harder than {\sc Vertex Cover} (under $\sP\neq \NP$).
We also prove new polynomial time results, some of which follow from a reduction to {\sc Vertex Cover} restricted to classes of probe graphs.
We first give a dichotomy on graphs where $G[T]$ is $H$-free. Namely, we show that 
{\sc Subset Vertex Cover} is polynomial-time solvable on graphs $G$, for which $G[T]$ is $H$-free, if $H=sP_1+tP_2$ and \NP-complete otherwise. Moreover, we prove that {\sc Subset Vertex Cover} is polynomial-time solvable for $(sP_1+P_2+P_3)$-free graphs and bounded mim-width graphs. By combining our new results with known results we obtain a partial complexity classification for {\sc Subset Vertex Cover} on $H$-free graphs.
\end{abstract}

\section{Introduction}\label{s-intro}

We consider a natural generalization of the classical {\sc Vertex Cover} problem: the {\sc Subset Vertex Cover} problem, introduced in~\cite{BJPP22}. 
Let $G=(V,E)$ be a graph and $T$ be a subset of $V$. 
A set $S\subseteq V$~is a {\it $T$-vertex cover} of $G$ if $S$ contains at least one end-vertex of every edge incident to a vertex of $T$. We note that~$T$ itself is a $T$-vertex cover. However, a graph may have much smaller $T$-vertex covers. For example, if $G$ is a star whose leaves form $T$, then the center of $G$ forms a $T$-vertex cover. We can now define the problem; see also Fig.~\ref{fig:svcillustrated}.

\problemdef{{\sc Subset Vertex Cover}}{A graph $G=(V,E)$, a subset $T\subseteq V$, and a positive integer $k$.}{Does $G$ have a $T$-vertex cover $S$ with $|S|\leq k$?} 

\noindent
If we set $T=V$, then we obtain the {\sc Vertex Cover} problem. Hence, as {\sc Vertex Cover} is \NP-complete, so is {\sc Subset Vertex Cover}.

\begin{figure}[b]
\centering
\includegraphics[scale=0.8]{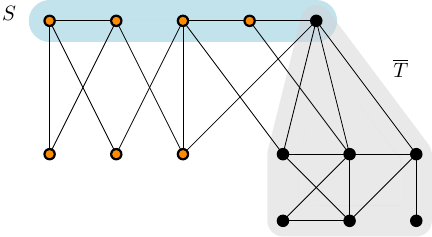}
\caption{An instance $(G,T,k)$ of \svc{}, where $T$ consists of the orange vertices, together with a solution $S$ (a $T$-vertex cover of size~$5$). Note that $S$ consists of four vertices of $T$ and one vertex of $\overline{T}=V\setminus T$.}\label{fig:svcillustrated}
\end{figure}

To obtain a better understanding of the complexity of an \NP-complete graph problem, we may restrict the input to some special graph class. In particular, {\it hereditary} graph classes, which are the classes closed under vertex deletion, have been studied intensively for this purpose. It is readily seen that a graph class~${\cal G}$ is hereditary if and only if ${\cal G}$ is characterized by a unique minimal set of forbidden induced subgraphs~${\cal F}_G$. Hence, for a systematic study, it is common to first consider the case where ${\cal F}_{\cal G}$ has size~$1$. This is also the approach we follow in this paper.
So, for a graph $H$, we set ${\cal F}_{\cal G}=\{H\}$ for some graph~$H$ and consider the class of $H$-free graphs (graphs that do not contain $H$ as an induced subgraph). We now consider the following research question:

\medskip
\noindent
{\it For which graphs $H$ is {\sc Subset Vertex Cover}, restricted to $H$-free graphs, still \NP-complete and for which graphs $H$ does it become polynomial-time solvable?}

\medskip
\noindent
We will also address two open problems posed in~\cite{BJPP22} (see Section~\ref{s-pre} for any undefined terminology):

\begin{itemize}
\item [Q1.] What is the complexity of {\sc Subset Vertex Cover} for claw-free graphs?
\item [Q2.] Is {\sc Subset Vertex Cover} \NP-complete for $P_t$-free graphs for some $t$?
\end{itemize}

\noindent
The first question is of interest, as {\sc Vertex Cover} is polynomial-time solvable even on $rK_{1,3}$-free graphs for every $r\geq 1$~\cite{BM18}, where $rK_{1,3}$ is the disjoint union of $r$ claws (previously this was known for $rP_3$-free graphs~\cite{Lo17} and $2P_3$-free graphs~\cite{LM12}).
The second question is of interest due to some recent quasi-polynomial-time results. Namely, Gartland and Lokshtanov~\cite{GL20} proved that for every integer~$t$, {\sc Vertex Cover} can be solved in $n^{O(\log^3n)}$-time for $P_t$-free graphs. Afterwards, 
Pilipczuk, Pilipczuk and Rz{\k a}\.{z}ewski~\cite{PPR21} improved the running time to $n^{O(\log^2n)}$ time. 
Even more recently, Gartland et al.~\cite{GLMPPR} extended the results of~\cite{GL20,PPR21} from $P_t$-free graphs to $H$-free graphs where every connected component of $H$ is a path or a subdivided claw.

Gr\"otschel, Lov\'asz, and Schrijver~\cite{GLS84} proved that {\sc Vertex Cover} can be solved in polynomial time for the class of perfect graphs. The class of perfect graphs is a rich graph class, which includes well-known graph classes, such as bipartite graphs and (weakly) chordal graphs. 

Before we present our results, we first briefly discuss the relevant literature.

\subsection{Existing Results}\label{s-existing}

Whenever {\sc Vertex Cover} is \NP-complete for some graph class~${\cal G}$, then so is the more general problem {\sc Subset Vertex Cover}. Moreover, {\sc Subset Vertex Cover} can be polynomially reduced to {\sc Vertex Cover}: given an instance $(G,T,k)$ of the former problem, remove all edges not incident to a vertex of $T$ to obtain an instance $(G',k)$ of the latter problem. Hence, we obtain:

\begin{proposition}\label{p-easy}
The problems {\sc Vertex Cover} and {\sc Subset Vertex Cover} are polynomially equivalent for every graph class closed under edge deletion.
\end{proposition}

\noindent
For example, the class of bipartite graphs is closed under edge deletion and {\sc Vertex Cover} is polynomial-time solvable on bipartite graphs. Hence,
by Proposition~\ref{p-easy}, {\sc Subset Vertex Cover} is polynomial-time solvable on bipartite graphs. 
However, a class of $H$-free graphs is only closed under edge deletion if $H$ is a complete graph, and {\sc Vertex Cover} is \NP-complete even for triangle-free graphs~\cite{Po74}. This means that there could still exist graphs $H$ such that {\sc Vertex Cover} and {\sc Subset Vertex Cover} behave differently if the former problem is (quasi)polynomial-time solvable on $H$-free graphs.
The following well-known result of Alekseev~\cite{Al82} restricts the structure of such graphs~$H$.

\begin{theorem}[\cite{Al82}]\label{t-hard}
For every graph $H$ that contains a cycle or a connected component with two vertices of degree at least~$3$, {\sc Vertex Cover}, and thus {\sc Subset Vertex Cover}, is \NP-complete on $H$-free graphs.
\end{theorem}

\noindent
Due to Theorem~\ref{t-hard} and the result of Gartland et al.~\cite{GLMPPR}, every graph $H$ is now either classified as a quasi-polynomial case or \NP-hard case for {\sc Vertex Cover}. For {\sc Subset Vertex Cover} the situation is much less clear. So far, only one positive result for $H$-free graphs is known, due to Brettell et al.~\cite{BJPP22}.
	
\begin{theorem}[\cite{BJPP22}]\label{t-sp1p4}
For every $s\geq 0$, {\sc Subset Vertex Cover} is polynomial-time solvable on $(sP_1+P_4)$-free graphs.
\end{theorem}

\noindent
We notice one more result for {\sc Subset Vertex Cover}, which can be obtained by making a 
connection to 
the concept of \emph{probe graphs}, introduced by Zhang et al.~\cite{Z94} in the context of genome research. Suppose $\mathcal{G}$ is a class of graphs. The class of probe graphs $\mathcal{G}_p$ of~$\mathcal{G}$ contains all graphs $G$ such that there exist an independent set $N$ in $G$ and a set of edges $F\subseteq {N \choose 2}$ such that $G + F \in \mathcal{G}$.
The intuition is that the vertices of $P = V(G) \setminus N$ are the probes, for which there is structural information, and the vertices of $N$ are the non-probes, for which there is no concrete information except that the ``actual graph'' $G + F$ belongs to some graph class, achieved by adding some edges between vertices of $N$.

We recall that an instance $(G, T, k)$ of \svc{} is reducible to an instance $(G', k)$ of \vc{}, where $G'$ is the graph obtained from $G$ by deleting all edges between vertices not in $T$. In other words,
we can solve \svc{} in polynomial time on a graph class~${\cal G}$ if we can solve \vc{} in polynomial time on ${\cal G}_p$.
Chang et al.~\cite{CKKLP05} observed that probe split graphs are perfect, and Golumbic and Lipshteyn~\cite{GL04} showed that probe interval graphs are probe chordal graphs, and that probe chordal graphs are perfect. As \vc{} is polynomial-time solvable on perfect graphs~\cite{GLS84}, it follows that \svc{} can be solved in polynomial time on chordal graphs.

\begin{theorem}[\cite{GL04,GLS84}]\label{t-chordal}
{\sc Subset Vertex Cover} can be solved in polynomial time on chordal graphs.
\end{theorem}

\subsection{Related Work}\label{s-existing2}

Subset variants of classic graph transversal problems are widely studied, also in the context of $H$-free graphs. Indeed, 
Brettell et al.~\cite{BJPP22} needed Theorem~\ref{t-sp1p4} as an auxiliary result in complexity studies for {\sc Subset Feedback Vertex Set} and {\sc Subset Odd Cycle Transversal} restricted to $H$-free graphs. The first problem is to decide for a graph $G=(V,E)$, subset $T\subseteq V$ and integer $k$, if $G$ has a set $S$ of size at most $k$ such that $S$ contains a vertex of every cycle that intersects $T$. The second problem is similar but replaces ``cycle'' by ``cycle of odd length''. Brettell et al.~\cite{BJPP22} proved that both these subset transversal problems are polynomial-time solvable on $(sP_1+P_3)$-free graphs for every $s\geq 0$.
They also showed that {\sc Odd Cycle Transversal} is polynomial-time solvable for $P_4$-free graphs and \NP-complete for split graphs, which form a subclass of $2P_2$-free graphs, whereas \NP-completeness for {\sc Subset Feedback Vertex Set} on split graphs was shown by Fomin et al.~\cite{FHKPV14}. Recently, Paesani et al.~\cite{PPR22b} extended the result of~\cite{BJPP22} for {\sc Subset Feedback Vertex Set} from $(sP_1+P_3)$-free graphs to $(sP_1+P_4)$-free graphs for every integer $s\geq 0$. If $H$ contains a cycle or claw, \NP-completeness for both subset transversal problems follows from corresponding results for {\sc Feedback Vertex Set}~\cite{Mu17b,Po74} and {\sc Odd Cycle Transversal}~\cite{CHJMP18}. 

Combining all the above results leads to 
a dichotomy for {\sc Subset Feedback Vertex Set} and a partial classification for {\sc Subset Odd Cycle Transversal} 
 (see also~\cite{BJPP22,PPR22b}). 
 Here, we write $F\ssi G$ if $F$ is an induced subgraph of $G$.

\begin{theorem}\label{t-dicho2}
For a graph~$H$, {\sc Subset Feedback Vertex Set} on $H$-free graphs is polynomial-time solvable if 
$H\ssi sP_1+P_4$ for some $s\geq 0$, and \NP-complete otherwise.
\end{theorem}

\begin{theorem}\label{t-dicho3}
For a graph~$H\neq sP_1+P_4$ for some $s\geq 1$, {\sc Subset Odd Cycle Transversal} on $H$-free graphs is polynomial-time solvable if 
$H=P_4$ or $H\ssi sP_1+P_3$ for some $s\geq 0$, and \NP-complete otherwise.
\end{theorem}

\noindent
We note that neither {\sc Subset Feedback Vertex Set} nor {\sc Subset Odd Cycle Transversal} restricted to some graph class~${\cal G}$ can be reduced to their classical counterparts on the probe graph class~${\cal G}_p$ by removing edges between any pair of vertices that both do not belong to $T$. So, in this sense these two subset transversal problems are different in nature than {\sc Subset Vertex Cover}.

\subsection{Our Results}

In Section~\ref{s-hard} we prove two new hardness results, using the same basis reduction, which may have a wider applicability. We first answer Q1 by proving that {\sc Subset Vertex Cover} is \NP-complete even for subcubic planar line graphs of triangle-free graphs, or equivalently, subcubic planar $(\mbox{claw},\mbox{diamond})$-free graphs~\cite{MT03}. We then answer Q2 by proving that {\sc Subset Vertex Cover} is \NP-complete even for $2$-unipolar graphs, which are $2P_3$-free (and thus $P_7$-free).

Our hardness results show a sharp contrast with {\sc Vertex Cover}, which can be solved in polynomial time for both weakly chordal graphs~\cite{GLS84} and $rK_{1,3}$-free graphs for every $r\geq 1$~\cite{BM18}. Hence, {\sc Subset Vertex Cover} may be harder than {\sc Vertex Cover} for a graph class closed under vertex deletion (if $\sP\neq \NP$). 
This is in contrast to graph classes closed under edge deletion (see Proposition~\ref{p-easy}).

In Section~\ref{s-hard} we also prove that {\sc Subset Vertex Cover} is \NP-complete for inputs $(G,T,k)$ if the subgraph $G[T]$ of $G$ induced by $T$ is $P_3$-free. On the other hand, our first positive result, shown in Section~\ref{s-poly}, shows that the problem is polynomial-time solvable if $G[T]$ is $sP_2$-free for any $s\geq 2$.
In Section~\ref{s-poly} we also prove that 
{\sc Subset Vertex Cover} can be solved in polynomial time for $(sP_1+P_2+P_3)$-free graphs for every $s\geq 1$. Our positive results generalize known results for {\sc Vertex Cover}. 
Recall that 
{\sc Subset Vertex Cover} is polynomial-time solvable for split graphs. Note that this also follows from our first result and contrasts our \NP-completeness result for $2$-unipolar graphs, which are generalized split, $2P_3$-free, and weakly chordal.

Combining our new results with Theorem~\ref{t-sp1p4} gives us a partial classification and a dichotomy, both of which are proven in Section~\ref{s-dicho}.

\begin{theorem}\label{t-dicho}
For a graph~$H\neq rP_1+sP_2+P_3$ for any $r\geq 0$, $s\geq 2$; $rP_1+sP_2+P_4$ for any $r\geq 0$, $s\geq 1$; or $rP_1+sP_2+P_t$ for any $r\geq 0$, $s\geq 0$, $t\in \{5,6\}$,
{\sc Subset Vertex Cover} on $H$-free graphs is polynomial-time solvable if 
$H\ssi sP_1+P_2+P_3$, $sP_2$, or $sP_1+P_4$ for some $s\geq 1$, and \NP-complete otherwise.
\end{theorem}

\begin{theorem}\label{t-dicho4}
For a graph $H$, {\sc Subset Vertex Cover} on instances $(G,T,k)$, where $G[T]$ is $H$-free, is polynomial-time solvable if $H\ssi sP_2$ for some $s\geq 1$, and \NP-complete otherwise.
\end{theorem}

\noindent
Theorems~\ref{t-dicho2}--\ref{t-dicho4} show that {\sc Subset Vertex Cover} on $H$-free graphs can be solved in polynomial time for infinitely more graphs $H$ than {\sc Subset Feedback Vertex Set} and {\sc Subset Odd Cycle Transversal}. This is in line with the behaviour of the corresponding original (non-subset) problems. 

In Section~\ref{s-mim} we discuss our final new result, which states that {\sc Subset Vertex Cover} is polynomial-time solvable on every graph class of bounded mim-width, such as the class of circular-arc graphs.
In Section~\ref{s-con} we discuss some directions for future work, which naturally originate from the above results.

\section{Preliminaries}\label{s-pre}

Let $G=(V,E)$ be a graph. The {\it degree} of a vertex $u\in V$ is the size of its {\it neighbourhood} $N(u)=\{v\; |\; uv\in E\}$. We say that $G$ is {\it subcubic} if every vertex of $G$ has degree at most~$3$. An independent set $I$ in $G$ is {\it maximal} if there exists no independent set $I'$ in $G$ with $I\subsetneq I'$. Similarly, a vertex cover $S$ of $G$ is {\it minimal} if there exists no vertex cover $S'$ in $G$ with $S'\subsetneq S$.
For a graph $H$ we write $H\ssi G$ if $H$ is an {\it induced} subgraph of $G$, that is, $G$ can be modified into $H$ by a sequence of vertex deletions. If $G$ does not contain $H$ as an induced subgraph, $G$ is {\it $H$-free}. For a set of graphs~${\cal H}$, $G$ is ${\cal H}$-free if $G$ is $H$-free for every $H\in {\cal H}$. If ${\cal H}=\{H_1,\ldots,H_p\}$ for some $p\geq 1$, we also write that $G$ is {\it $(H_1,\ldots,H_p)$-free}. 

The {\it line graph} of a graph $G=(V,E)$ is the graph~$L(G)$ that has vertex set $E$ and an edge between two vertices $e$ and $f$ if and only if $e$ and $f$ share a common end-vertex in $G$. The {\it complement} $\overline{G}$ of a graph $G=(V,E)$ has vertex set $V$ and an edge between two vertices $u$ and $v$ if and only if $uv\notin E$. 

For two vertex-disjoint graphs $F$ and $G$, the {\it disjoint union} $F+G$ is the graph $(V(F)\cup V(G), E(F)\cup E(G))$. We denote the disjoint union of $s$ copies of the same graph $G$ by $sG$.
A {\it linear forest} is a disjoint union of one or more paths.

Let $C_s$ be the cycle on $s$ vertices; $P_t$ the path on $t$ vertices; $K_r$ the complete graph on $r$ vertices; and $K_{1,r}$ the star on $(r+1)$ vertices. The graph~$C_3=K_3$ is the {\it triangle}; the graph $K_{1,3}$ the {\it claw}, and the graph $\overline{2P_1+P_2}$ is the {\it diamond} (so the diamond is obtained from the $K_4$ after deleting one edge).
The {\it subdivision} of an edge $uv$ replaces $uv$ with a new vertex $w$ and edges $uw$, $wv$.
A {\it subdivided claw} is obtained from the claw by subdividing each of its edges zero or more times.

A graph is {\it chordal} if it has no induced $C_s$ for any $s\geq 4$. A graph is
{\it weakly chordal} if it has no induced~$C_s$ and no induced~$\overline{C_s}$ for any $s\geq 5$.
A cycle $C_s$ or an anti-cycle $\overline{C_s}$ is {\it odd} if it has an odd number of vertices.
By the Strong Perfect Graph Theorem~\cite{CRST06}, a graph is {\it perfect} if it has no odd induced~$C_s$ and no odd induced $\overline{C_s}$ for any $s\geq 5$. Every chordal graph is weakly chordal, and every weakly chordal graph is perfect.
A graph $G=(V,E)$ is {\it unipolar} if $V$ can be partitioned into two sets~$V_1$ and~$V_2$, where~$G[V_1]$ is a complete graph and $G[V_2]$ is a disjoint union of complete graphs. If every connected component of $G[V_2]$ has size at most~$2$, then $G$ is {\it $2$-unipolar}.
Unipolar graphs form a subclass of {\it generalized split graphs}, which are the graphs that are unipolar or their complement is unipolar. 
It can also readily be checked that every $2$-unipolar graph is weakly chordal 
(but not necessarily chordal, as evidenced by $G=C_4$).
 
For an integer~$r$, a graph $G'$ is an {\it $r$-subdivision} of a graph $G$ if $G'$ can be obtained from $G$ by subdividing every edge of $G$ $r$ times, that is, by replacing each edge $uv\in E(G)$ with a path from $u$ to $v$ of length~$r+1$.
 
\section{NP-Hardness Results}\label{s-hard}

In this section we prove our hardness results for {\sc Subset Vertex Cover}, using the following notation.
Let $G$ be a graph with an independent set $I$. We say that we {\it augment} $G$ by adding a (possibly empty) set $F$ of edges between some pairs of vertices of $I$. We call the resulting graph an {\it $I$-augmentation} of $G$.

The following lemma forms the basis for our hardness gadgets.

\begin{lemma}\label{l-augment}
Every vertex cover of a graph $G=(V,E)$ with an independent set~$I$ is a $(V\setminus I)$-vertex cover of every $I$-augmentation of $G$, and vice versa.
\end{lemma}

\begin{proof}
Let $G'$ be an $I$-augmentation of $G$. Consider a vertex cover $S$ of $G$. For a contradiction, assume that $S$ is not a $(V\setminus I)$-vertex cover of $G'$.
Then $G'-S$ must contain an edge $uv$ with at least one of $u,v$ belonging to $V\setminus I$. As $G-S$ is an independent set, $uv$ belongs to $E(G')\setminus E(G)$ implying that both $u$ and $v$ belong to $I$, a contradiction.

Now consider a $(V\setminus I)$-vertex cover $S'$ of $G'$. For a contradiction, assume that $S'$ is not a vertex cover of $G$. Then $G-S'$ must contain an edge $uv$ (so $uv\in E$). As $G'$ is a supergraph of $G$, we find that $G'-S'$ also contains the edge $uv$. As $S'$ is a $(V\setminus I)$-vertex cover of $G'$, both $u$ and $v$ must belong to $I$. As $uv\in E$, this contradicts the fact that $I$ is an independent set.
\end{proof}

\noindent
To use Lemma~\ref{l-augment} we need one other lemma, which follows directly from an observation due to Poljak~\cite{Po74}. 

\begin{lemma}[\cite{Po74}]\label{l-po}
For an integer~$r$, a graph $G$ with $m$ edges has an independent set of size~$k$ if and only if the $2r$-subdivision of $G$ has an independent set of size $k+rm$.
\end{lemma}

\noindent
We are now ready to prove our first two hardness results. Recall that a graph is $(\mbox{claw},\mbox{diamond})$-free if and only if it is a line graph of a triangle-free graph. Hence, the result in particular implies \NP-hardness of {\sc Subset Vertex Cover} on line graphs. 
Recall also that we denote the claw and diamond by $K_{1,3}$ and $\overline{2P_1+P_2}$, respectively.

\begin{figure}[t]
\centering
\includegraphics[scale=0.8]{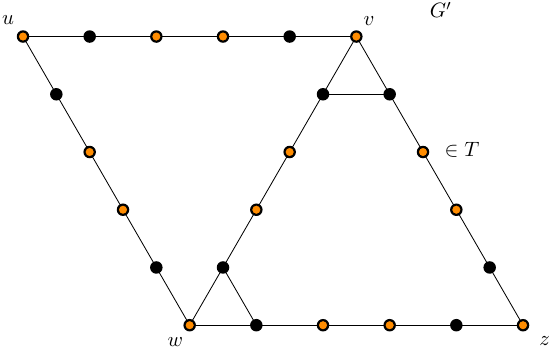}
\caption{The graph $G'$ from Theorem~\ref{t-hard1}, where $T = V\setminus W$ consists of the orange vertices.}\label{f-hard1}
\end{figure}

\begin{theorem}\label{t-hard1}
{\sc Subset Vertex Cover} is \NP-complete on $(K_{1,3},\overline{2P_1+P_2})$-free subcubic planar graphs.
\end{theorem}

\begin{proof}
We reduce from {\sc Vertex Cover}, which is \NP-complete even on cubic planar graphs~\cite{Mo01}.
As an $n$-vertex graph has a vertex cover of size at most $k$ if and only if it has an independent set of size at least $n-k$, we find that {\sc Vertex Cover} 
is \NP-complete even on subcubic planar graphs that are $4$-subdivisions due to 
an application of Lemma~\ref{l-po} with $r=2$ (note that subdividing an edge preserves both maximum degree and planarity).
So, let $(G,k)$ be an instance of {\sc Vertex Cover}, where $G=(V,E)$ is a subcubic planar graph that is a $4$-subdivision of some cubic planar graph~$G^*$, and $k$ is an integer. 

In $G$, we let $U=V(G^*)$ and $W$ be the subset of $V(G)\setminus U$ that consists of all neighbours of vertices of $U$. Note that $W$ is an independent set in $G$. We construct a $W$-augmentation $G'$ as follows; see also Figure~\ref{f-hard1}. For every vertex $u\in U$ of degree~$3$ in $G$, we pick two arbitrary neighbours of $u$ (which both belong to $W$) and add an edge between them. It is readily seen that $G'$ is $(K_{1,3},\overline{2P_1+P_2})$-free, planar and subcubic. 
By Lemma~\ref{l-augment}, it holds that $G$ has a vertex cover of size at most $k$ if and only if $G'$ has a $(V\setminus W)$-vertex cover of size at most $k$.
\end{proof}

\begin{theorem}\label{t-hard2}
{\sc Subset Vertex Cover} is \NP-complete for instances $(G,T,k)$, for which $G$ is $2$-unipolar and $G[T]$ is a disjoint union of edges.
\end{theorem}

\begin{proof}
We reduce from {\sc Vertex Cover}. 
So, let $(G,k)$ be an instance of {\sc Vertex Cover}, where $G=(V,E)$ is a graph and $k$ is an integer. 
By Lemma~\ref{l-po} (take $r=1$) we may assume that $G$ is a $2$-subdivision of a graph~$G^*$. Note that $V(G^*)$ is an independent set in $G$. We construct a $V(G^*)$-augmentation $G'$ of $G$ by changing $V(G^*)$ into a clique; see also Figure~\ref{f-hard2}. It is readily seen that $G'$ is $2$-unipolar. We set $T:=V\setminus V(G^*)$, so $G[T]$ is a disjoint union of edges. By Lemma~\ref{l-augment}, it holds that $G$ has a vertex cover of size at most~$k$ if and only if $G'$ has a $T$-vertex cover of size at most $k$.
\end{proof}

\begin{remark}
	It can be readily checked that $2$-unipolar graphs are $(2C_3,C_5,C_6,C_3+P_3,2P_3,\overline{P_6},\overline{C_6})$-free graphs, and thus are $2P_3$-free and weakly chordal.
\end{remark}

\begin{figure}[t]
\centering
\includegraphics[scale=0.8]{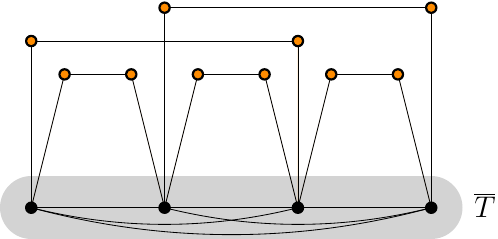}
\caption{The graph $G'$ from Theorem~\ref{t-hard2}, where the orange vertices form $T = V\setminus V(G^*)$.}\label{f-hard2}
\end{figure}

\section{Polynomial-Time Results}\label{s-poly}

In this section, we prove our polynomial-time results for instances $(G,T,k)$ where either $G$ is $H$-free or only $G[T]$ is $H$-free. The latter type of results are stronger, but only hold for graphs $H$ with smaller connected components.

For completeness, we state K\H{o}nig's theorem here, which has as an immediate consequence that \vc{} can be solved in polynomial time on bipartite graphs, which is constructively shown by Bondy and Murty~\cite{BondyMurtyGraphTheory}.
\begin{theorem}[K\H{o}nig's Theorem~\cite{Konig31}]\label{t-Konig}
	In a bipartite graph, the maximum number of edges in a matching is equal to the minimum number of vertices in a vertex cover.
\end{theorem}

We start with the case where $H=sP_2$ for some $s\geq 1$.
For this case we need the following two well-known results. The \emph{delay} of an enumeration algorithm
 is the maximum of the time taken before the first output and that between any pair of consecutive outputs. 
	
\begin{theorem}[\cite{BY89}]\label{t-by}
For every constant~$s\geq 1$, the number of maximal independent sets of an $sP_2$-free graph on $n$ vertices is at most $n^{2s}+1$.
\end{theorem}
	
\begin{theorem}[\cite{TAS77}]\label{t-tias}
For every constant~$s\geq 1$, it is possible to enumerate all maximal independent sets of an $sP_2$-free graph $G$ on $n$ vertices and $m$ edges with a delay of $O(nm)$.
\end{theorem}

\noindent
We now prove that {\sc Subset Vertex Cover} is polynomial-time solvable on instances $(G,T,k)$ where $G[T]$ is $sP_2$-free. The idea behind the algorithm is to remove any edges between vertices in $V\setminus T$, as these edges are irrelevant. As a consequence, we may leave the graph class, but this is not necessarily an obstacle. For example, if $G[T]$ is a complete graph, or $T$ is an independent set, we can easily solve the problem. Both cases are generalized by the result below.

\begin{theorem}\label{t-sp2}
For every $s\geq 1$, {\sc Subset Vertex Cover} can be solved in polynomial time on instances $(G,T,k)$ for which $G[T]$ is $sP_2$-free.
\end{theorem}

\begin{proof}
Let $s\geq 1$, and let $(G,T,k)$ be an instance of {\sc Subset Vertex Cover} where $G=(V,E)$ is a graph such that $G[T]$ is $sP_2$-free.
 Let $G'=(V,E')$ be the graph obtained from $G$ after removing every edge between two vertices of $V\setminus T$, so $G'[V\setminus T]$ is edgeless. We observe that $G$ has a $T$-vertex cover of size at most~$k$ if and only if $G'$ has a $T$-vertex cover of size at most~$k$. Moreover, $G'[T]$ is $sP_2$-free, and we can obtain $G'$ in $O(|E(G)|)$ time.
Hence, from now on, we consider the instance $(G',T,k)$.
 
We first prove the following two claims, see Figure~\ref{fig:sp2WR} for an illustration.

\begin{figure}[t]
	\centering
	\includegraphics[scale=0.9]{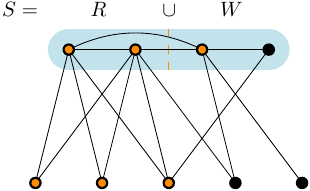}
	\caption{An example of the $2P_2$-free graph $G'$ of the proof of Theorem~\ref{t-sp2}. Here, $T$ consists of the orange vertices. A solution $S$ can be split up into a minimal vertex cover $R$ of $G'[T]$ and a vertex cover $W$ of $G[V\setminus R]$.}\label{fig:sp2WR}
\end{figure}

\begin{claim}\label{clm:1}
A subset $S\subseteq V(G')$ is a $T$-vertex cover of $G'$ if and only if $S=R\cup W$ for a minimal vertex cover $R$ of $G'[T]$ and a vertex cover $W$ of $G'[V\setminus R]$.
\end{claim}

\noindent
{\it Proof.}
We prove Claim~\ref{clm:1} as follows. Let $S\subseteq V(G')$. First assume that $S$ is a $T$-vertex cover of $G'$. Let $I=V\setminus S$. As $S$ is a $T$-vertex cover, $T\cap I$ is an independent set. Hence, $S$ contains a minimal vertex cover $R$ of $G'[T]$. As $G'[V\setminus T]$ is edgeless, $S$ is a vertex cover of $G'$, or in other words, $I$ is an independent set. In particular, this means that $S \setminus R$ is a vertex cover of $G'[V\setminus R]$.

Now assume that $S=R\cup W$ for a minimal vertex cover $R$ of $G'[T]$ and a vertex cover $W$ of $G'[V\setminus R]$. For a contradiction, suppose that $S$ is not a $T$-vertex cover of $G'$.
Then $G'-S$ contains an edge $uv\in E'$, where at least one of $u,v$ belongs to $T$. First suppose that both $u$ and $v$ belong to $T$. As $R$ is a vertex cover of $G'[T]$, at least one of $u$, $v$ belongs to $R\subseteq S$, a contradiction. Hence, exactly one of $u,v$ belongs to $T$, say $u\in T$ and $v\in V\setminus T$, so in particular, $v\notin R$. As $R\subseteq S$, we find that $u\notin R$. Hence, both $u$ and $v$ belong to $V\setminus R$. As $W$ is a vertex cover of $G'[V\setminus R]$, this means that at least one of $u,v$ belongs to $W\subseteq S$, a contradiction.
This proves the claim. \dia

\begin{claim}\label{clm:2}
For every minimal vertex cover $R$ of $G'[T]$, the graph $G'[V\setminus R]$ is bipartite.
\end{claim}

\noindent
{\it Proof.}
We prove Claim~\ref{clm:2} as follows. As $R$ is a vertex cover of $G'[T]$, we find that $T\setminus R$ is an independent set. As $G'[V\setminus T]$ is edgeless by construction of $G'$, this means that $G'[V\setminus R]$ is bipartite with partition classes $T\setminus R$ and $V\setminus T$. \dia

\medskip
\noindent
We are now ready to give our algorithm. We enumerate the minimal vertex covers of $G'[T]$. For every minimal vertex cover~$R$, we compute a minimum vertex cover $W$ of $G'[V\setminus R]$. In the end, we return the smallest $S=R\cup W$ that we found. 

The correctness of our algorithm follows from Claim~\ref{clm:1}. It remains to analyze the running time. As $G'[T]$ is $sP_2$-free, we can enumerate all maximal independent sets $I$ of $G'[T]$ and thus all minimal vertex covers $R=T\setminus I$ of $G'[T]$ in $(n^{2s}+1) \cdot O(nm)$ time due to Theorems~\ref{t-by} and~\ref{t-tias}. For a minimal vertex cover~$R$, the graph $G'[V\setminus R]$ is bipartite by Claim~\ref{clm:2}. Hence, we can compute a minimum vertex cover $W$ of $G'[V\setminus R]$ in polynomial time by applying K\H{o}nig's Theorem (Theorem~\ref{t-Konig}). We conclude that the total running time is polynomial.
\end{proof}

\noindent
For our next result (Theorem~\ref{t-sp1p2p3}) we need two known results as lemmas.

\begin{lemma}[\cite{BJPP22}]\label{lem:P1add}
If {\sc Subset Vertex Cover} is polynomial-time solvable on $H$-free graphs for some $H$, then it is so on $(H+P_1)$-free graphs.
\end{lemma}

\begin{lemma}[\cite{BM18}]\label{l-lk13}
For every $r\geq 1$, {\sc Vertex Cover} is polynomial-time solvable on $rK_{1,3}$-free graphs.
\end{lemma}

\noindent
We are now ready to prove our second polynomial-time result.

\begin{theorem}\label{t-sp1p2p3}
For every integer~$s$, {\sc Subset Vertex Cover} is polynomial-time solvable on $(sP_1+P_2+P_3)$-free graphs.
\end{theorem}

\begin{proof}
Due to Lemma~\ref{lem:P1add}, we can take $s=0$, so we only need to give a polynomial-time algorithm for $(P_2+P_3)$-free graphs.
Hence, let $(G,T,k)$ be an instance of {\sc Subset Vertex Cover}, where $G=(V,E)$ is a $(P_2+P_3)$-free graph. 

First, we compute a minimum vertex cover~$S_{\mbox{vc}}$ of $G$. As $G$ is $(P_2+P_3)$-free, and thus $2K_{1,3}$-free, this takes polynomial time by Lemma~\ref{l-lk13}.

We now set out to compute a minimum $T$-vertex cover $S$ of $G$ that is not a vertex cover of $G$. Then $G-S$ must contain an edge between two vertices in $V \setminus T$. We branch by considering all $O(n^2)$ options of choosing this edge. For each chosen edge $uv$ we set out to construct a smallest $T$-vertex cover $S_{uv}$ of $G$ that does not contain $u$ nor $v$. Observe that we thus must add every neighbour of $u$ or $v$ that belongs to $T$ to $S_{uv}$, as otherwise $G-S_{uv}$ would contain an edge $ut$ or $vt$ for some vertex $t\in T$. Hence, we have $N(\{u,v\}) \cap T \subseteq S_{uv}$.

Let $T'=T\setminus N(\{u,v\})$ consist of all vertices of $T$ that are neither adjacent to $u$ nor to $v$. As $G$ is $(P_2+P_3)$-free and $uv\in E$, we find that $G[T']$ is $P_3$-free, and thus $G[T']$ is a disjoint union of complete graphs.
We call a connected component of $G[T']$ {\it large} if it has at least two vertices; else we call it {\it small} (so every large connected component of $G[T']$ is a complete graph on at least two vertices and every small connected component of $G[T']$ is an isolated vertex).

\medskip
\noindent
{\bf Case 1.} The graph $G[T']$ has at most two large connected components.\\
Let $D_1$ and $D_2$ be the large connected components of $G[T']$ (if they exist).
As $V(D_1)$ and $V(D_2)$ are cliques in $G[T']$, at most one vertex of $D_1$ and at most one vertex of $D_2$ can belong to $G-S_{uv}$ if $S_{uv}$ is to become a $T$-vertex cover. We branch by considering all $O(n^2)$ options of choosing at most one vertex of $D_1$ and at most one vertex of $D_2$ to be these vertices. For each choice of vertices we do as follows. We add all other vertices of $D_1$ and $D_2$ to a set $S^*_{uv}$.
Let $T^*$ be the set of vertices of $T'$ that we have not added to $S^*_{uv}$. Then $T^*$ consist of all the vertices of the small connected components of $G[T']$ and at most one vertex of each of the at most two large connected components of $G[T']$. Hence, $T^*$ is an independent set.

We delete every edge between any two vertices in $V \setminus T$. Now the graph $G^*$ induced by $T^*\cup (V\setminus T)$ is bipartite, namely with partition classes $T^*$ and $V\setminus T$. It remains to compute a minimum vertex cover~$S^*$ of $G^*$. This can be done in polynomial time by applying K\H{o}nig's Theorem. We let $S^*_{uv}$ consist of $S^*$ together with all vertices of $T$ that we had added to $S^*_{uv}$ already. 

Finally, let $S_{uv}$ be a smallest such set $S^*_{uv}$ (found over all $O(n^2)$ branches) together with $N(\{u,v\}) \cap T$, so $S_{uv}$ is a smallest $T$-vertex cover of $G$ that does not contain $u$ nor $v$. This completes Case~1.

\medskip
\noindent
{\bf Case 2.} The graph $G[T']$ has at least three large connected components.\\
Let $D_1,\ldots, D_p$, for some $p\geq 3$, be the large connected components of $G[T']$. 
Let $A$ consist of all the vertices of the small connected components of $G[T']$.

We start by considering each vertex~$w\in V\setminus T$ with one of the following properties:

\begin{itemize}
\item [1.] for some $i$, $w$ has a neighbour and a non-neighbour in $D_i$; or
\item [2.] for some $i,j$ with $i\neq j$, $w$ has a neighbour in $D_i$ and a neighbour in $D_j$; or
\item [3.] for some $i$, $w$ has a neighbour in $D_i$ and a neighbour in $A$.
\end{itemize}

\noindent
Let $W$ be the set of all vertices of $V\setminus T$ that satisfy at least one of the Properties 1--3 (note that $W$ does not contain $u$ nor $v$).
We say that a vertex $w\in V \setminus T$ is {\it semi-complete} to some $D_i$ if $w$ is adjacent to all vertices of $D_i$ except at most one.
\begin{figure}[t]
	\centering
	\includegraphics[scale=0.8]{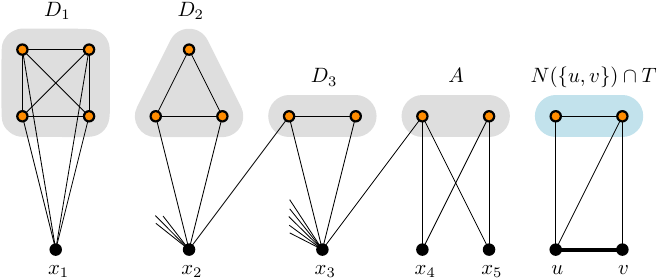}
	\caption{An illustration of the graph $G$ in the proof of Case~2 of Theorem~\ref{t-sp1p2p3}, where $T$ consists of the orange vertices, and $p=3$. Edges in $G[V\setminus T]$ are not drawn, and for $x_2$ and $x_3$ some edges are partially drawn. Vertices $x_1, x_4, x_5,u,v$ do not belong to $W$, as they do not satisfy one of the Properties 1--3, whereas $x_2$ belongs to $W$, as $x_2$ satisfies Property 1 for $D_2$ (and also Property 2 for $D_2$ and $D_3$), and $x_3$ belongs to $W$, as $x_3$ satisfies Property~3 for $D_3$.}\label{fig:p2p3}
\end{figure}
We show the following claim. See Figure~\ref{fig:p2p3} for an illustration.

\begin{claim}\label{c-claim}
Every vertex $w\in W$ is semi-complete to $D_i$ for every $i\in\{1,\ldots,p\}$.
\end{claim}

\noindent
{\it Proof.}
Let $w\in W$.
First, assume $w$ satisfies Property~1. Let $x$ and $y$ be vertices of some $D_i$, say $D_1$, such that $wx\in E$ and $wy\notin E$. For a contradiction, assume $w$ is not semi-complete to some $D_j$. Hence, $D_j$ contains vertices $y'$ and $y''$, such that $wy'\notin E$ and $wy''\notin E$. If $j\geq 2$, then $\{y',y'',w,x,y\}$ induces a $P_2+P_3$ (as $D_1$ and $D_j$ are complete graphs). This contradicts that $G$ is $(P_2+P_3)$-free. Hence, $w$ is semi-complete to every $V(D_j)$ with $j\geq 2$. 
Now suppose $j=1$. As $p\geq 3$, the graphs $D_2$ and $D_3$ exist. As $w$ is semi-complete to every $D_j$ for $j\geq 2$ and every $D_j$ is large, there exist vertices $x'\in V(D_2)$ and $x''\in V(D_3)$ such that $wx'\in E$ and $wx''\in E$. However, now $\{y',y'',x',w,x''\}$ induces a $P_2+P_3$, a contradiction. 

Now assume $w$ satisfies Property~2, say $w$ is adjacent to $x_1\in V(D_1)$ and to $x_2\in V(D_2)$. Suppose $w$ is not semi-complete to some $D_j$. If $j\geq 3$, then the two non-neighbours of $w$ in $D_j$, together with $x_1,w,x_2$, form an induced $P_2+P_3$, a contradiction. Hence, $w$ is semi-complete to every $D_j$ for $j\geq 3$. If $j\in \{1,2\}$, say $j=1$, then let $y,y'$ be two non-neighbours of $w$ in $D_1$ and let $x_3$ be a neighbour of $w$ in $D_3$. Now, $\{y,y',x_2,w,x_3\}$ induces a $P_2+P_3$, a contradiction. Hence, $w$ is semi-complete to $D_1$ and $D_2$ as well.

Finally, assume $w$ satisfies Property~3, say $w$ is adjacent to $z\in A$ and $x_1\in V(D_1)$. If $w$ is not semi-complete to $D_j$ for some $j\geq 2$, then two non-neighbours of $w$ in $D_j$, with $z,w,x_1$, form an induced $P_2+P_3$, a contradiction. Hence, $w$ is semi-complete to every $D_j$ with $j\geq 2$. As before, by using a neighbour of $w$ in $D_2$ and one in $D_3$, we find that $w$ is also semi-complete to~$D_1$.
This completes the proof of Claim~\ref{c-claim}. \dia

\medskip
\noindent
We now first consider the possible situation where $S_{uv}$ will not contain some vertex $w\in W$. We branch by considering all $O(n)$ options for choosing such a vertex $w$. For each chosen vertex~$w$, we do as follows. Let $T^w=N(w)\cap T'$ be the set of neighbours of $w$ in $T'$ (recall that $T'=T\setminus N(\{u,v\})$). So, all of the vertices of $T^w$ must belong to $S_{uv}$. By Claim~\ref{c-claim}, we find that $T'\setminus T^w$ does not contain two vertices from the same large connected component of $G[T']$. Hence, $T'\setminus T^w$ is an independent set. We delete any edge between two vertices from $V\setminus T$, so $V\setminus T$ becomes an independent set as well. We now compute, in polynomial time by K\H{o}nig's Theorem, a minimum vertex cover in the resulting bipartite graph with partition classes $T'\setminus T^w$ and $V\setminus T$.
We let $S^w_{uv}$ be the set that is the union of this vertex cover, $T^w$, and $(N\{u,v\}\cap T)$. By construction, $S^w_{uv}$ is a smallest vertex cover of $G$ that does not contain any vertex of $\{u,v,w\}$. After processing all of the $O(n)$ branches, we keep a smallest set $S^w_{uv}$, which we denote by~$S_{uv}^*$.

We are left to examine the possible situation where $S_{uv}$ contains every vertex of $W$. Let $G'$ be the subgraph obtained from $G$ by removing every vertex of 
$W\cup (N(\{u,v\})\cap T) \cup \{u,v\}$. We will now describe how to compute in polynomial time a smallest $T'$-vertex cover $S_{uv}'$ of $G'$ (recall that $T'=T\setminus N(\{u,v\})$).

We start with considering the connected component~$D_1'$ of~$G'$ that contains (all) the vertices from~$D_1$. As no vertex from $W$ belongs to $G'$ by definition, $D_1'$ contains no vertices from $V\setminus T$ satisfying Property~2 or~3. Hence, $D_1'$ contains no vertices from $A$ or from any $D_j$ with $j\geq 2$ either. So it holds that 
$$V(D_1')\cap T=V(D_1).$$
First suppose that $D_1'=D_1$. 
As $D_1'=D_1$ is a complete graph, 
we add all vertices of $D_1$ except for one arbitrary vertex of $D_1$ to $S_{uv}'$.
Now suppose that there exists a vertex $x$ in $V(D_1')\setminus V(D_1)$. As $D_1'$ is connected, we may assume without loss of generality that $x$ has a neighbour in $D_1$. Consequently, $x$ is complete to $D_1$, as $x$ does not belong to $W$ and thus does not satisfy Property 1. This implies that we must put at least $|V(D_1)|$ vertices from~$D_1'$ 
in~$S_{uv}'$, so we might just as well put every vertex of $D_1$ in~$S_{uv}'$. As $V(D_1')\cap T=V(D_1)$, we do not need to add any more vertices from $D_1'$ to $S_{uv}'$. 

We do the same as we did for $D_1$ for the connected components $D_2', \ldots, D_p'$ of $G'$ that contain the sets $V(D_2),\ldots, V(D_p)$, respectively. 

Now, it remains to consider the induced subgraph $F$ of $G'$ that consists of the connected components of $G'$ containing the vertices of $A$. Recall that $A$ is an independent set. We delete every edge between two vertices in $V\setminus T$, resulting in another independent set. This changes $F$ into a bipartite graph, and we can compute a minimum vertex cover $S^F_{uv}$ of $F$ in polynomial time due to K\H{o}nig's Theorem. We add $S^F_{uv}$ to $S_{uv}'$. In this way we computed the set $S_{uv}'$. We let $S_{uv}''=S_{uv}'\cup W\cup (N(\{u,v\})\cap T)$. By construction, $S_{uv}''$ is a smallest vertex cover $S_{uv}''$ of $G$ that contains $W$ but does not contain $u$ and $v$.

Now compare the size of $S_{uv}''$ with the size of $S_{uv}^*$, and pick one of smallest size as $S_{uv}$. This completes Case~2.

\medskip
\noindent
Finally, for each choice of the edge $uv \in E(G[V \setminus T])$, we consider the output $S_{uv}$, and take a smallest set found. We compare its size with the size of $S_{\mbox{vc}}$, again taking a smallest set as the final solution.

The correctness of our algorithm follows from the above description. The number of branches is $O(n^2)$ in Case~1 and $O(n)$ in Case~2. Hence, as there $O(n^2)$ vertex pairs $u,v$ to consider, the total number of branches is $O(n^4)$. As each branch takes polynomial time to process, this means that the total running time of our algorithm is polynomial. This completes our proof.
\end{proof}

\section{The Proof of Theorems~\ref{t-dicho} and~\ref{t-dicho4}}\label{s-dicho}

We first prove Theorem~\ref{t-dicho}, which we restate below.

\medskip
\noindent
{\bf Theorem~\ref{t-dicho} (restated).}
{\it For a graph~$H\neq rP_1+sP_2+P_3$ for any $r\geq 0$, $s\geq 2$; $rP_1+sP_2+P_4$ for any $r\geq 0$, $s\geq 1$; or $rP_1+sP_2+P_t$ for any $r\geq 0$, $s\geq 0$, $t\in \{5,6\}$,
{\sc Subset Vertex Cover} on $H$-free graphs is polynomial-time solvable if 
$H\ssi sP_1+P_2+P_3$, $sP_2$, or $sP_1+P_4$ for some $s\geq 1$, and \NP-complete otherwise.}

\begin{proof}
	Let $H$ be a graph not equal to $rP_1+sP_2+P_3$ for any $r\geq 0$, $s\geq 2$; $rP_1+sP_2+P_4$ for any $r\geq 0$, $s\geq 1$; or $rP_1+sP_2+P_t$ for any $r\geq 0$, $s\geq 0$, $t\in \{5,6\}$. If $H$ has a cycle, then we apply Theorem~\ref{t-hard}. Else, $H$ is a forest. If $H$ has a vertex of degree at least~$3$, then the class of $H$-free graphs contains all $K_{1,3}$-free graphs, and we apply Theorem~\ref{t-hard1}. Else, $H$ is a linear forest. If $H$ contains an induced $2P_3$, then we apply Theorem~\ref{t-hard2}. If not, then $H\ssi sP_1+P_2+P_3$, $sP_2$, or $sP_1+P_4$ for some $s\geq 1$. In the first case, apply Theorem~\ref{t-sp1p2p3}; in the second case Theorem~\ref{t-sp2}; and in the third case Theorem~\ref{t-sp1p4}.
\end{proof}

\noindent
We now prove Theorem~\ref{t-dicho4}, which we restate below.

\medskip
\noindent
{\bf Theorem~\ref{t-dicho4} (restated).}
{\it For a graph $H$, {\sc Subset Vertex Cover} on instances $(G,T,k)$, where $G[T]$ is $H$-free, is polynomial-time solvable if $H\ssi sP_2$ for some $s\geq 1$, and \NP-complete otherwise.}

\begin{proof}
	First note that any \textsf{NP}-completeness result of Theorem~\ref{t-dicho} also holds here, as $G[T]$ is $H$-free if the entire graph is $H$-free. We immediately get that \svc{} is \textsf{NP}-complete if $H$ is not a linear forest.
	Suppose $P_3\ssi H$. Observe that in the instance of Theorem~\ref{t-hard2}, $G[T]$ is a disjoint union of edges, so $G[T]$ is $P_3$-free, and hence, $H$-free. We see that \svc{} is \textsf{NP}-complete when $P_3\subseteq_i H$.
	In the remaining case, $H\ssi sP_2$ for some $s\geq 1$, and we apply Theorem~\ref{t-sp2}.
\end{proof}

\section{Graphs of Bounded Clique-Width and Bounded Mim-width}\label{s-mim}
In this section, we give a polynomial algorithm for {\sc Subset Vertex Cover} on graphs of bounded clique-width and of bounded mim-width.

We begin with a result regarding the clique-width of probe graphs. Recall that $\mathcal{G}_p$ denotes the class of probe graphs of some graph class $\mathcal{G}$ and that a graph $G$ belongs to $\mathcal{G}_p$ if there is some independent set $N$ of $G$ and a set of edges $F$ with end-vertices in $N$ such that $G + F \in \mathcal{G}$.
The clique-width $\cw(G)$ of a graph~$G$, introduced by Courcelle, Engelfriet and Rozenberg~\cite{CER93}, is the minimum number of labels needed to construct $G$ by means of the following four operations:
\begin{enumerate}
\item creating a new labeled vertex;
\item taking the disjoint union of two labeled graphs;
\item adding an edge between every vertex of label $i$ and every vertex of label~$j$, where $i \neq j$; and
\item changing the label of all vertices with label~$i$ to label~$j$ for some $j\geq i$.
\end{enumerate}
\noindent
Cographs have clique-width at most~$2$~\cite{CO00}.
The authors of~\cite{CCKLP12,CKKLP05} observed that probe cographs have clique-width at most~$4$. As {\sc Vertex Cover} is polynomial-time solvable on graphs of bounded clique-width~\cite{CMR00}, this observation directly implies Theorem~\ref{t-sp1p4} for $s=0$. In fact, the observation can be readily generalized, as we prove below for the sake of completeness. 

\begin{proposition}\label{p-cw}
Let $G = (V,E)$ be a graph, 
let $N$ be an independent set of $G$, and
let $F$ be a set of edges with both end-vertices in $N$.
Then we have $\cw(G) \leq 2\cw(G+F)$.
\end{proposition}
\begin{proof}
We adapt a sequence of operations that constructs $G+F$ using only $\cw(G+F)$ labels to a sequence of operations that constructs $G$ using $2\cw(G+F)$ labels.
If an operation is to create a new vertex $v$ with label $i$, then create $v$ and label it with $i_N$ if $v \in N$ and with $i_P$ otherwise.
If an operation is the disjoint union of two labeled graphs, then we keep this operation.
If an operation is adding an edge between every vertex of label $i$ and every vertex of label~$j$, where $i \neq j$, then we add an edge between every vertex of label $i_P$ and every vertex of label~$j_P$, every vertex of label $i_N$ and every vertex of label~$j_P$, and every vertex of label~$i_P$ and every vertex of label~$j_N$.
If an operation is renaming a label $i$ to $j$, then we rename $i_N$ to $j_N$ and~$i_P$ to~$j_P$.
Note that these operations construct $G$ because they include all edges from $G+F$ except the edges with both end-vertices in $N$. Clearly, we use $2\cw(G+F)$ labels, which proves the statement.
\end{proof}

We now introduce some new terminology to explain the notion of mim-width, which was first defined by Vatshelle~\cite{Va12}.
Let $G=(V,E)$ be a graph. For $X \subseteq V$, we use $2^X$ to denote its power set and $\overline{X}$ to denote $V \setminus X$. A set $M \subseteq E$ is a \emph{matching} in $G$ if no two edges in $M$ share an end-vertex. A matching $M$ is an \emph{induced matching} if no end-vertex of any edge $e \in M$ is adjacent to any end-vertex of an edge in $M\setminus \{e\}$. 
A \emph{rooted binary tree} is a rooted tree of which each node has degree~$1$ or~$3$, except for a distinguished node that has degree~$2$ and is the root of the tree. A \emph{rooted layout} $\calL = (L,\delta)$ of $G$ consists of a rooted binary tree $L$ and a bijection $\delta$ between $V$ and the leaves of $L$. For each node $x \in V(L)$, we let $L_x$ be the set of leaves that are a descendant of $x$ (including $x$ if $x$ is a leaf). We define $V_x$ as the corresponding set of vertices of $G$, that is, $V_x = \{ \delta(y) \mid y \in L_x \}$. 
For a set $A \subseteq V$, let $\mim_G(A)$ be the size of a maximum induced matching in the bipartite graph obtained from $G$ by removing all edges between vertices of $A$ and all edges between vertices of $\overline{A}=V\setminus A$. In other words, this is the bipartite graph $(A,\overline{A}, E \cap (A \times \overline{A}))$. The \emph{mim-width} $\mim_G(\calL)$ of a rooted layout $\calL = (L,\delta)$ of graph $G$ is the maximum over all $x \in V(L)$ of $\mim_G(V_x)$. If the graph in question is clear from the context, then we omit the subscript. The \emph{mim-width} $\mim(G)$ of $G$ is the minimum mim-width over all rooted layouts of $G$. See Figure~\ref{f-mimex} for an example.

\begin{figure}[ht]
	\centering
	\includegraphics[width=.8\textwidth]{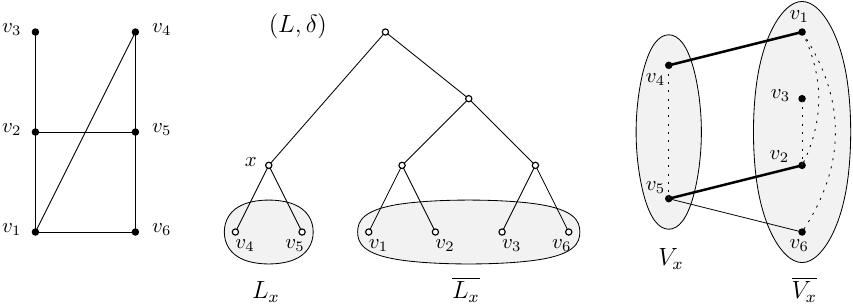}
	\caption{An example of a graph $G$ (left) with a rooted layout $\calL=(L,\delta)$ (middle), taken from~\cite{BGMPP22}. Note that $\mim(V_x)=2$ (right), and an easy check shows that in fact $\mim(\calL)=2$. The rooted layout ${\mathcal L}'$ obtained from $\calL$ by swapping 
$v_2$ and $v_5$ and swapping $v_3$ and $v_4$ yields $\mim(G)=1$.}\label{f-mimex}
\end{figure}

In general, it is not known if there exists a polynomial-time algorithm for computing a rooted layout~${\cal L}$ of a graph~$G$, such that $\mim({\cal L})$ is bounded by a function in the mim-width of $G$.
However, Belmonte and Vatshelle~\cite{BV13} 
showed that for several graph classes~${\cal G}$ of bounded mim-width, including interval graphs and permutation graphs, it is possible to find in polynomial time a rooted layout of a graph $G\in {\cal G}$ with mim-width equal to the mim-width of $G$.

We now prove the following result on the mim-width of probe graphs that is analogous to Proposition~\ref{p-cw} for clique-width.

\begin{proposition}\label{p-mim}
Let $G = (V,E)$ be a graph, let $N$ be an independent set of $G$, 
let $F$ be a set of edges with both ends in $N$,
and let $\calL = (L, \delta)$ be a rooted layout of $G+F$.
Then we have $\mim_G(\calL) \leq 2\mim_{G+F}(\calL)$ and, in particular, $\mim(G) \leq 2\mim(G+F)$.
\end{proposition}
\begin{proof}
As $G$ and $G+F$ have identical vertex sets, $\calL$ is a rooted layout of $G$ too.
Let $x \in V(L)$ be such that $\mim_G(V_x) = \mim_G(\calL)$,
and let $H$ denote the bipartite graph $(V_x, \overline{V_x}, E \cap (V_x \times \overline{V_x}))$.
In $H$, let $E_1$ denote the edges with one end-vertex in $V_x \cap N$,
let $E_2$ denote the edges with one end-vertex in $\overline{V_x} \cap N$,
and let $E_3$ denote the edges with no end-vertex in $N$.
Note that the sets $E_1$, $E_2$ and $E_3$ are a partition of $E(H)$ since $N$ is independent in $G$.
Lastly, let $M$ be a maximum induced matching of $H$.
In particular, $|M| = \mim_G(V_x)$.

By symmetry, we may assume that $|E_1 \cap M| \leq |E_2 \cap M|$.
The set $M' = M \cap (E_2 \cup E_3)$ is an induced matching of $H+F$ too.
On the one hand, we have $|M'| \geq |M| / 2$ by the assumption that $|E_2 \cap M| \geq |E_1 \cap M|$.
On the other hand, we have $|M'| \leq \mim_{G+F}(V_x)$ by definition.
Together, this implies
\[
\mim_G(\calL) = \mim_G(V_x) = |M| \leq 2|M'| \leq 2\mim_{G+F}(V_x) \leq 2\mim_{G+F}(\calL).
\]
This argument holds for any rooted layout $\calL$, so in particular, it holds for the layout $\calL$ such that $\mim_{G+F}(\calL) = \mim(G+F)$. For any layout $\calL$, it holds that $\mim(G) \leq \mim_G(\calL)$, and so we may conclude that $\mim(G) \leq 2\mim(G+F)$ holds. This completes the proof.
\end{proof}

Let $\mathcal{G}$ be a class of graphs of bounded mim-width.
Then, Proposition~\ref{p-mim} implies that $\mathcal{G}_p$ also has bounded mim-width.
Therefore, an algorithm for \svc{} on classes of bounded mim-width graphs is implied by results in the literature.
Suppose that $(G, T, k)$ is an instance of \svc{} is given together with a rooted layout $\calL = (L, \delta)$ of $G$.
Let $G'$ be the graph obtained from $G$ by deleting all edges not incident to a vertex of $T$.
Then $\calL$ is also a rooted layout of $G'$ and $\mim_{G'}(\calL) \leq 2\mim_{G}(\calL)$ by Proposition~\ref{p-mim}.
Now, solving the \svc{} instance $(G, T, k)$ amounts to solving the \vc{} instance $(G', k)$.
The dynamic programming algorithm for \vc{} of Bui-Xuan et al.~\cite{BTV13} has running time $O(n^4 \cdot \nec_1(T, \delta)^3)$ ($\nec_d(.)$ is defined below).
Belmonte and Vatshelle~\cite{BV13} showed that $\nec_d(A) \leq |A|^{d \cdot \mim_G(A)}$ for any $A \subseteq V(G)$.
Together, this implies an $O(n^{6\mim_{G}(\calL) + 4})$-time algorithm for \svc{} with a given rooted layout $\calL$ of $G$.
However, this is a weaker result than that of Theorem~\ref{t-svcmimresult}, which is attained through a direct algorithm that we give next.

We now introduce the notion of neighbour equivalence, which was first defined by Bui-Xuan et al.~\cite{BTV13}. 
Let $G=(V,E)$ be a graph on $n$ vertices. Let $A \subseteq V$ and $d \in \mathbb{N}^+$. We say that $X, W \subseteq A$ are \emph{$d$-neighbour equivalent} \wrt $A$, denoted $X \equiv^A_d W$, if $\min\{d, |X \cap N(v)|\} = \min\{d,|W \cap N(v)|\}$ for all $v \in \overline{A}$. Clearly, this is an equivalence relation. We let $\nec_d(A)$ denote the number of equivalence classes of $\equiv^A_d$.
For a rooted layout $(L, \delta)$ of $G$, we let $\nec_d(T, \delta)$ denote the maximum of $\nec_d(V_x)$ over all $x \in V(L)$.
For each $X \subseteq A$, let $\rep^A_d(X)$ denote the lexicographically smallest set $R \subseteq A$ such that $R \equiv^A_d X$ and $|R|$ is minimum. This is called the \emph{representative} of $X$. We use $\calR^A_d = \{ \rep^A_d(X) \mid X \subseteq A \}$. Note that $|\calR^A_d| \geq 1$, as the empty set is always a representative. The following lemma allows us to work efficiently with representatives.

\begin{lemma}[Bui-Xuan et al.~\cite{BTV13}]\label{l-rep}
	It is possible to compute in time $O(\nec_d(A) \cdot n^2 \log (\nec_d(A)))$, the set $\calR^A_d$ and a data structure that given a set $X \subseteq A$, computes $\rep^A_d(X)$ in $O(|A| \cdot n \log (\nec_d(A)))$ time.
\end{lemma}

\noindent
We are now ready to give an explicit polynomial-time algorithm for {\sc Subset Vertex Cover} on graphs of bounded mim-width. Our algorithm is inspired by the algorithm of Bui-Xuan et al.~\cite{BTV13} for {\sc Independent Set} and of Bergougnoux et al.~\cite{BPT22} for {\sc Subset Feedback Vertex Set}. Our presentation of the algorithm follows the presentation form in Bergougnoux et al.~\cite{BPT22}. In fact, we solve the complementary problem. Given a graph $G = (V,E)$ with a rooted layout $\calL = (L,\delta)$, a set $\Ti \subseteq V$, and a weight function $\we$ on its vertices, we find a maximum-weight \Tis on $G$. Our goal is to use a standard dynamic programming algorithm. However, the size of the table that we would need to maintain by a naive approach is too large. Instead, we work with representatives of the sets in our table. We show that we can reduce the table size so that it is bounded by the square of the number of $1$-neighbour equivalence classes.

First, we define a notion of equivalence between elements of our dynamic programming table. Given a set $\Ti \subseteq V$, a set $X \subseteq V$ is a \emph{\Tis} if in $G[X]$ there is no edge incident on any vertex of $\Ti \cap X$. Note that $X$ is a \Tis if and only if $\overline{X}$ is a $\Ti$-vertex cover.

\begin{definition}\label{d-simt}
	Let $X,W \subseteq V_x$ be \Tiss. We say that $X$ and $W$ are \emph{equivalent}, denoted by $X \sim_\Ti W$, if $X \cap \Ti \equiv^{V_x}_1 W \cap \Ti$ and $X \setminus \Ti \equiv^{V_x}_1 W \setminus \Ti$.
\end{definition}

\noindent
We now prove the following lemma.

\begin{lemma} \label{l-sim}
	For every $Y \subseteq \overline{V_x}$ and every \Tiss $X,W \subseteq V_x$ such that $X \sim_\Ti W$, it holds that $X \cup Y$ is a \Tis if and only if $W \cup Y$ is a \Tis.
\end{lemma}
\begin{proof}
	By symmetry, it suffices to prove one direction. Suppose that $X \cup Y$ is a \Tis, but $W \cup Y$ is not. Note that $X$ and $W$ are \Tiss by definition and that $Y$ must be a \Tis as well, because $X \cup Y$ is. Hence, the fact that $W \cup Y$ is not a \Tis implies there is an edge $uv \in E(G)$ for which:
	\begin{enumerate}
		\item $u \in W \cap \Ti, v \in Y \cap \Ti$,
		\item $u \in W \cap \Ti, v \in Y \setminus \Ti$, or
		\item $u \in W \setminus \Ti, v \in Y \cap \Ti$.
	\end{enumerate}
	In the first case, since $v \in Y \cap \Ti$ has a neighbour in $W \cap \Ti$, note that $\min\{1,|(W \cap \Ti) \cap N(v)|\} = 1$. Since $X \cap \Ti \equiv^{V_x}_1 W \cap \Ti$ by the assumption that $X \sim_\Ti W$, it follows that $\min\{1,|(X \cap \Ti) \cap N(v)|\} = 1$. Hence, there is an edge from $v \in Y \cap \Ti$ to $X \cap \Ti$, contradicting that $X \cup Y$ is a \Tis. 
	
	The second case is analogous to the first case. The third case is also analogous, but uses that $X \setminus \Ti \equiv^{V_x}_1 W \setminus \Ti$.
\end{proof}

\noindent
We now introduce a final definition.

\begin{definition}
	For every $\calA \subseteq 2^{V_x}$ and $Y \subseteq \overline{V_x}$, let
	$$\best(\calA, Y) = \max\{ \we(X) \mid X \in \calA \mbox{\ and\ } X \cup Y \mbox{\ is a \Tis} \}.$$
	Given $\calA, \calB \subseteq 2^{V_x}$, we say that $\calB$ \emph{represents} $\calA$ if $\best(\calA,Y) = \best(\calB,Y)$ for every $Y \subseteq \overline{V_x}$.
\end{definition}

\noindent
We use the above definition in our next lemma.

\begin{lemma} \label{l-reduce}
	Given a set $\calA \subseteq 2^{V_x}$, we can compute $\calB \subseteq \calA$ that represents $\calA$ and has size at most $\nec_1(V_x)^2$ in $O(|\calA| \cdot n^2 \log (\nec_1(V_x)) + \nec_1(V_x) \cdot n^2 \log (\nec_1(V_x)) + \nec_1(V_x)^2)$ time.
\end{lemma}
\begin{proof}
	We obtain $\calB$ from $\calA$ as follows: for all sets in $\calA$ that are equivalent under $\sim_\Ti$, maintain only a set $X$ that is a \Tis for which $\we(X)$ is maximum. Note that if among a collection of equivalent sets, there is no \Tis, then no set is maintained. By construction, $\calB$ has size at most $\nec_1(V_x)^2$, because there are at most $\nec_1(V_x)^2$ equivalence classes of $\sim_T$ by Definition~\ref{d-simt}.
	
	We now prove that $\calB$ represents $\calA$. Let $Y \subseteq \overline{V_x}$. Note that $\best(\calB,Y) \leq \best(\calA, Y)$, because $\calB \subseteq \calA$. Hence, if there is no $X \in \calA$ such that $X \cup Y$ is a \Tis, then $\best(\calB,Y) = \best(\calA,Y) = -\infty$. So assume otherwise, and let $W \in \calA$ satisfy $\we(W) = \best(\calA,Y)$. This means that $W \cup Y$ is a \Tis and in particular, $W$ is a \Tis. By the construction of $\calB$, there is a $X \in \calB$ that is a \Tis with $X \sim_\Ti W$ and $\we(X) \geq \we(W)$. By Lemma~\ref{l-sim}, $X \cup Y$ is a \Tis. Hence, $\best(\calB, Y) \geq \we(X) \geq \we(W) = \best(\calA, Y)$. It follows that $\best(\calB,Y) = \best(\calA,Y)$ and thus $\calB$ represents $\calA$.
	
	For the running time, note that we can implement the algorithm by maintaining a table indexed by pairs of representatives of the $1$-neighbour equivalence classes. Note that a pair of representatives uniquely identifies an equivalence class of $\sim_T$. By Lemma~\ref{l-rep}, we can compute the indices in $O(\nec_1(V_x) \cdot n^2 \log (\nec_1(V_x)))$ time, which gives us both the complete set of representatives and the data structure to compute representatives. Then for each $X \in \calA$, we can compute its representatives in $O(|V_x| \cdot n \log (\nec_1(V_x)))$ time and check whether it is a \Tis in $O(n^2)$ time. Creation of the table and returning $\calB$ takes $O(\nec_1(V_x)^2)$ time. Hence, the total running time is $O(|\calA| \cdot n^2 \log (\nec_1(V_x)) + \nec_1(V_x) \cdot n^2 \log (\nec_1(V_x)) + \nec_1(V_x)^2)$.
\end{proof}

\noindent
We are now ready to prove the following result.

\begin{theorem}\label{t-theo}
	Let $G$ be a graph on $n$ vertices with a rooted layout $(L, \delta)$. We can solve {\sc Subset Vertex Cover} in $O(\sum_{x \in V(L)} (\nec_1(V_x))^4 \cdot n^2 \log (\nec_1(V_x)))$ time.
\end{theorem}
\begin{proof}
	It suffices to find a maximum-weight \Tis of $G$. For every node $x \in V(L)$, we aim to compute a set 
	$\calA_x \subseteq 2^{V_x}$ of \Tiss such that $\calA_x$ represents $2^{V_x}$ and has size at most $p(x) := \nec_1(V_x)^2+1$. Letting $r$ denote the root of $L$, we then return the set in $\calA_r$ of maximum weight. Since $\calA_r$ represents $2^{V_r}$, this is indeed a maximum-weight \Tis of $G$.
	
	We employ a bottom-up dynamic programming algorithm to compute $\calA_x$. If $x$ is a leaf with $V_x = \{v\}$, then set $\calA_x = \{\emptyset, \{v\}\}$. Clearly, $\calA_x$ represents $2^{V_x}$ and has size at most $p(x)$. So now suppose $x$ is an internal node with children $a,b$. For any $\calA,\calB \subseteq 2^{V(G)}$, let $\calA \otimes \calB = \{X \cup Y \mid X \in \calA, Y \in \calB\}$. Now let $\calA_x$ be equal to the result of the algorithm of Lemma~\ref{l-reduce} applied to $\calA_a \otimes \calA_b$. Then, indeed, $|\calA_x| \leq p(x)$. Using induction, it remains to show the following for the correctness proof:
	
	\begin{claim}\label{c-claim2}
		If $\calA_a$ and $\calA_b$ represent $2^{V_a}$ and $2^{V_b}$ respectively, then the computed set $\calA_x$ represents $2^{V_x}$.
	\end{claim}
	
\noindent
{\it Proof.} 
We prove Claim~\ref{c-claim2} as follows.
If $\calA_a \otimes \calA_b$ represents $2^{V_x}$, then by Lemma~\ref{l-reduce} and the transitivity of the `represents' relation, it follows that $\calA_x$ represents $2^{V_x}$. So it suffices to prove that $\calA_a \otimes \calA_b$ represents $2^{V_x}$. Let $Y \subseteq \overline{V_x}$. Note that 
		$$\begin{array}{rcl}
			\best(\calA_a \otimes \calA_b, Y) & = & \max\{ \we(X) + \we(W) \mid X \in \calA_a, W \in \calA_b, \\ & & \qquad\qquad X \cup W \cup Y \mbox{\ is a \Tis} \} \\
			& = & \max\{\best(\calA_a, W \cup Y) +\we(W) \mid W \in \calA_b \}.
		\end{array}$$
		Note that $\best(\calA_a, W \cup Y) = \best(2^{V_a}, W \cup Y)$, as $\calA_a$ represents $2^{V_a}$, and thus $\best(\calA_a \otimes \calA_b, Y) = \best(2^{V_a} \otimes \calA_b, Y)$. Using a similar argument, we can then show that $\best(2^{V_a} \otimes \calA_b, Y) = \best(2^{V_a} \otimes 2^{V_b}, Y)$. Since $2^{V_x} = 2^{V_a} \otimes 2^{V_b}$, it follows that $\best(\calA_a \otimes \calA_b, Y) = \best(2^{V_x}, Y)$ and thus $\calA_a \otimes \calA_b$ represents $2^{V_x}$. 
This completes the proof of Claim~\ref{c-claim2}.\dia
	
\medskip
\noindent	
Finally, we prove the running time bound. Using induction, it follows that $|\calA_a \otimes \calA_b| \leq p(x)^2$ for any internal node $x$ with children $a,b$. Hence, $\calA_a \otimes \calA_b$ can be computed in $O(p(x)^2 \cdot n)$ time. Then, $\calA_x$ can be computed in $O(p(x)^2 \cdot n^2 \log (\nec_1(V_x)) + \nec_1(V_x) \cdot n^2 \log (\nec_1(V_x)) + \nec_1(V_x)^2) = O((\nec_1(V_x))^4 \cdot n^2 \log (\nec_1(V_x)))$ time by Lemma~\ref{l-reduce}. 
\end{proof}

\noindent
It was shown by Belmonte and Vatshelle~\cite{BV13} that $\nec_d(A) \leq |A|^{d \cdot \mim(A)}$. Combining their result with Theorem~\ref{t-theo} immediately yields the following.

\begin{theorem}\label{t-svcmimresult}
	Let $G$ be a graph on $n$ vertices with a rooted layout $\calL = (L, \delta)$. Then {\sc Subset Vertex Cover} can be solved in $O(\mim_G(\calL) \cdot n^{4\mim_G(\calL) + 3} \cdot \log(n))$ time.
\end{theorem}

\noindent
The following corollary is now immediate from the fact that circular-arc graphs have constant mim-width and a rooted layout of constant mim-width can be computed in polynomial time~\cite{BV13}.

\begin{corollary}\label{cor:intcircarc}
	{\sc Subset Vertex Cover} can be solved in polynomial time on circular-arc graphs.
\end{corollary}

\noindent
As interval graphs are circular-arc, Corollary~\ref{cor:intcircarc} also holds for interval graphs, but we also note that the result for interval graphs already follows from Theorem~\ref{t-chordal}.

\section{Conclusions}\label{s-con}

Apart from giving a dichotomy for {\sc Subset Vertex Cover} restricted to instances $(G,T,k$) where $G[T]$ is $H$-free (Theorem~\ref{t-dicho4}), we gave a partial classification of {\sc Subset Vertex Cover} for $H$-free graphs (Theorem~\ref{t-dicho}). Our partial classification resolved two open problems from the literature and showed that for some hereditary graph classes, {\sc Subset Vertex Cover} is computationally harder than {\sc Vertex Cover} (if $\sP\neq \NP$). This is in contrast to the situation for graph classes closed under edge deletion. Hence, {\sc Subset Vertex Cover} is worth studying on its own, instead of only as an auxiliary problem (as in~\cite{BJPP22}). In order to complete the classification of {\sc Subset Vertex Cover} for $H$-free graphs it remains to solve precisely the following open cases.

\begin{open}
Determine the computational complexity of {\sc Subset Vertex Cover} on $H$-free graphs, where
\begin{itemize}
\item 
$H=sP_2+P_3$ for $s\geq 2$; or 
\item 
$H=sP_2+P_4$ for $s\geq 1$; or
\item 
$H=sP_2+P_t$ for $s\geq 0$ and $t\in \{5,6\}$.
\end{itemize}
\end{open}

\noindent
Brettell et al.~\cite{BJPP22} asked what the complexity of {\sc Subset Vertex Cover} is for $P_5$-free graphs. In contrast, {\sc Vertex Cover} is polynomial-time solvable even for $P_6$-free graphs~\cite{GKPP22}. 
However, the open cases where $H=sP_2+P_t$ ($s\geq 1$ and $t\in \{4,5,6\}$) are even open for {\sc Vertex Cover} on $H$-free graphs (though a quasipolynomial-time algorithm is known~\cite{GL20,PPR21}). So for those cases we may want to first restrict ourselves to {\sc Vertex Cover} instead of {\sc Subset Vertex Cover}, 
or aim for a quasipolynomial-time algorithm first.

We also note that our polynomial-time algorithms for {\sc Subset Vertex Cover} for $sP_2$-free graphs and $(P_2+P_3)$-free graphs can easily be adapted to work for {\sc Weighted Subset Vertex Cover} for $sP_2$-free graphs and $(P_2+P_3)$-free graphs. In this more general problem variant, each vertex~$u$ is given some positive weight $w(u)$, and the question is whether there exists a $T$-vertex cover $S$ with weights $w(S)=\sum_{u\in S}w(u)\leq k$. In contrast, Papadopoulos and Tzimas~\cite{PT20} proved that {\sc Weighted Subset Feedback Vertex Set} is \NP-complete for $5P_1$-free graphs, whereas {\sc Subset Feedback Vertex Set} is polynomial-time solvable even for $(sP_1+P_4)$-free graphs for every $s\geq 1$~\cite{PPR22b} (see also Theorem~\ref{t-dicho2}). The hardness construction of Papadopoulos and Tzimas~\cite{PT20} can also be used to prove that {\sc Weighted Odd Cycle Transversal} is \NP-complete for $5P_1$-free graphs~\cite{BJP22}, while {\sc Subset Odd Cycle Transversal} is polynomial-time solvable even for $(sP_1+P_3)$-free graphs for every $s\geq 1$~\cite{BJPP22} (see also Theorem~\ref{t-dicho3}).

Finally, we recall from Theorem~\ref{t-chordal} that \svc{} can be solved in polynomial time on chordal graphs by using a reduction to \vc{} on perfect graphs and applying the linear programming method of~\cite{GLS84}. 
However, it is not known if this algorithm is strongly-polynomial.
In contrast, we gave a purely combinatorial algorithm for probe split graphs in Theorem~\ref{t-sp2} and for probe interval graphs in Corollary~\ref{cor:intcircarc}.
This makes the following question interesting.

\begin{open}\label{o-chordal}
Give a combinatorial algorithm for {\sc Subset Vertex Cover} on the class of chordal graphs.
\end{open}
 
\noindent
A standard approach for {\sc Vertex Cover} on chordal graphs is dynamic programming over the clique tree of a chordal graph. However, a naive dynamic programming algorithm over the clique tree does not work for {\sc Subset Vertex Cover}. This is because we may need to remember an exponential number of subsets of a bag (clique) and the bags can have arbitrarily large size. 
Hence, to solve Open Problem~\ref{o-chordal} new ideas are needed.

\bibliographystyle{plain}

\end{document}